%% file: FGSimple.tex
\documentclass{amsart}

\usepackage{amssymb}
\usepackage{mathrsfs}
\usepackage[colorlinks=true,linkcolor=blue,citecolor=magenta]{hyperref}

\input{macros}

\title[]
{Finitely generated infinite simple groups of homeomorphisms of the real line}

\thanks{This research has been supported by a Swiss national science foundation ``Ambizione" grant of the second author.
The second author would like to thank Air France for its hospitality on a flight during which a portion of this paper was written.
The second author would also like to thank SNU and KAIST for their hospitality during a visit to Korea in the course of which some of this work was completed.
The authors would like to thank Justin Moore, Michele Triestino, Nicol\'{a}s Matte Bon, Sang-hyun Kim, Thomas Koberda, Matt Brin and Nicolas Monod for helpful remarks on the exposition.}

\author{James Hyde} 
\address{Mathematical Institute,
University of St. Andrews,
Scotland.}\email{jameshydemaths@gmail.com}
\author{Yash Lodha}
\address{Institute of Mathematics, EPFL, Lausanne, Switzerland.}\email{yash.lodha@epfl.ch}



\begin{document}

\dedicatory{} 

\begin{abstract}
We construct examples of finitely generated infinite simple groups of homeomorphisms of the real line.
Equivalently, these are examples of finitely generated simple left (or right) orderable groups.
This answers a well known open question of Rhemtulla from $1980$ concerning the existence of such groups.
In fact, our construction provides a family of continuum many isomorphism types of groups with these properties. 
\end{abstract}

\maketitle

\section{Introduction}

In this article we answer the following question.

\begin{question}\label{question}
Is there a finitely generated infinite simple group of homeomorphisms of the real line?
\end{question}

It is clear that a group that positively answers Question \ref{question} must indeed be a subgroup of $\textup{Homeo}^+(\mathbf{R})$,
which is the group of orientation preserving homeomorphisms of the real line.
Recall that a countable group embeds in $\textup{Homeo}^+(\mathbf{R})$ if and only if it is \emph{left orderable}, i.e. it admits a total order that is invariant under left multiplication.
(See \cite{Navas}). Therefore the above question can be restated as follows. 

\begin{question}\label{questionreformulation}
(Rhemtulla $1980$) Is there a finitely generated simple left orderable group?
\end{question}

(This appears as question $16.50$ in the ``Kourkova Notebook (No. $19$): Unsolved problems in group theory"  \cite{Kourkova}, 
Question $1.67$ in the book ``Ordered Groups and Topology" by Clay and Rolfsen \cite{ClayRolfsen}, Question vii in ``Groups, orders, and laws" by Navas \cite{Navas2} and
Question $9$ on page $265$ in ``Ordered Algebraic Structures; ed. Martinez, Jorge" \cite{Martinez}.)

Countable simple groups that are not finitely generated are abundant in \newline $\textup{Homeo}^+(\mathbf{R})$.
Some natural examples include commutator subgroups of certain groups of piecewise linear and piecewise projective homeomorphisms.
For instance, the commutator subgroup of Thompson's group $F$ and certain generalisations.
(See \cite{BurilloLodhaReeves} for more examples).
In fact, the second author together with Kim and Koberda (in \cite{KimKoberdaLodha}) constructed a continuum family of pairwise nonisomorphic countable simple subgroups of $\textup{Homeo}^+(\mathbf{R})$.
These groups occur naturally as commutator subgroups of the so called \emph{chain groups}, which are also constructed in \cite{KimKoberdaLodha}, and for natural reasons do not admit finite generating sets.

Moreover, there are several examples of finitely presented infinite simple groups of homeomorphisms of the circle,
such as Thompson's group $T$ and related examples.
(See \cite{Lodha} for a recent construction by the second author.) 
However, none of these groups admit non-trivial actions on the real line by homeomorphisms.

A classical obstruction to Question \ref{question} is the so called \emph{Thurston Stability theorem},
which says that any group of $C^1$-diffeomorphisms of an interval of the type $[r,s)$ or $(r,s]$ is \emph{locally indicable}, i.e. every finitely generated subgroup admits a homomorphism onto $\mathbf{Z}$.
Therefore, given a finitely generated simple group $G<\textup{Homeo}^+(\mathbf{R})$, the extention of the action to $[-\infty,\infty)$ or $(-\infty,\infty]$
must not be conjugate to an action by $C^1$-diffeomorphisms.
It is also known that bi-orderable simple groups cannot be finitely generated (see the discussion before Question $1.67$ in \cite{ClayRolfsen}),
and that finitely generated amenable left orderable groups cannot be simple (see \cite{WitteMorris}).

Another obstruction is the so called \emph{germ homomorphism} onto the groups of germs at $\pm \infty$, which is non trivial for many naturally occurring examples of subgroups of $\textup{Homeo}^+(\mathbf{R})$.
Finally, the standard methods for proving simplicity for subgroups of $\textup{Homeo}^+(\mathbf{R})$ (for instance Higman's simplicity criterion, see \cite{BurilloLodhaReeves}) require that the group is expressed as an increasing union of \emph{compactly supported} subgroups,
and hence the resulting group cannot be finitely generatable.
Providing a proof of simplicity that does not rely on the group being expressed in this way is the key technical challenge in approaching Question \ref{question},
and hence the novelty of our construction.

In this article we introduce a systematic construction of finitely generated simple subgroups of $\textup{Homeo}^+(\mathbf{R})$.
Recall that the group $\textup{PL}^+([0,1])$ is the group of orientation preserving piecewise linear homeomorphisms of $[0,1]$.
Our construction is obtained from gluing two different non standard actions of certain subgroups of $\textup{PL}^+([0,1])$ in a controlled manner.
The construction takes as an input a certain \emph{quasi-periodic} labelling $\rho$ of the set $\frac{1}{2}\mathbf{Z}$ which is a map $$\rho:\frac{1}{2}\mathbf{Z}\to \{a,b,a^{-1},b^{-1}\}$$ that satisfies a certain set of axioms.
Such labellings exist and are easy to construct explicitly (see Lemma \ref{quasi-periodicLabellings}).
For each such labelling $\rho$, we construct an explicit group action $G_{\rho}<\textup{Homeo}^+(\mathbf{R})$ and show that it satisfies the following.

\begin{thm}\label{main}
Let $\rho$ be a quasi-periodic labelling of $\frac{1}{2}\mathbf{Z}$.
Then the group $G_{\rho}$ is a finitely generated simple subgroup of $\textup{Homeo}^+(\mathbf{R})$.
\end{thm}

In general different quasi-periodic labellings may produce isomorphic groups.
However, the above construction is generalised in Section \ref{continuum} to provide an uncountable family of pairwise nonisomorphic groups with the desired features.
In addition to a quasi-periodic labelling, the generalized construction uses as an additional input a real $\alpha\in (0,1)\setminus \mathbf{Z}[\frac{1}{2}]$. To each such pair, we associate a group $G_{\rho,\alpha}$ and show the following.

\begin{thm}\label{continuumsimple}
The family $$\mathcal{S}=\{G_{\rho,\alpha}\mid \rho\text{ is a quasi-periodic }\text{labelling and }\alpha\in (0,1)\setminus \mathbf{Z}[\frac{1}{2}]\}$$
consists of finitely generated simple subgroups of $\textup{Homeo}^+(\mathbf{R})$ of continuum many isomorphism types. 
\end{thm}

Equivalently, we obtain a strong positive answer to Question \ref{questionreformulation}.

\begin{cor}
There exist continuum many isomorphism types of finitely generated simple left orderable groups.
\end{cor}

Finally, we remark that the groups we construct are also interesting test cases for the following well known open question (see Question $7.5$ in \cite{BekkaHarpeValette} and Question $3$ in \cite{NavasICM}).

\begin{question}\label{open}
Is there an infinite group of homeomorphisms of the real line with Property $(\mathrm{T})$?
\end{question}

In this direction, it is natural to inquire whether variations of our construction can provide such examples.
$G_{\rho}$ does not satisfy some of the known obstructions to $(\mathrm{T})$ for groups acting on $1$-manifolds.
This includes the criterion established in the recent work of the second author with Matte Bon and Triestino (see Theorem $1.1$ and Corollary $1.3$ in  \cite{LodhaMatteBonTriestino}).
In fact, the groups $G_{\rho}$ act by countably singular $C^{\infty}$-diffeomorphisms on the \emph{non compact} manifold $\mathbf{R}$.
This is precisely the situation in which the criterion in \cite{LodhaMatteBonTriestino} (and even the cocycle in \cite{LodhaMatteBonTriestino})
fails to establish that $G_{\rho}$ does not have $(\mathrm{T})$.

\section{Preliminaries}

All actions will be right actions, unless otherwise specified.
Given a group action $G<\textup{Homeo}^+(\mathbf{R})$ and a $g\in G$,
we denote by $Supp(g)$, or the \emph{open support}, as the set $$Supp(g)=\{x\in \mathbf{R}\mid x\cdot g\neq x\}$$
Note that $Supp(g)$ is an open set, and that $\mathbf{R}$ can be replaced by another $1$-manifold.
A homeomorphism $f:[0,1]\to [0,1]$ is said to be \emph{compactly supported in} $(0,1)$
if $\overline{Supp(f)}\subset (0,1)$.
Similarly, a homeomorphism $f:\mathbf{R}\to \mathbf{R}$ is said to be \emph{compactly supported in} $\mathbf{R}$
if $\overline{Supp(f)}$ is a compact interval in $\mathbf{R}$.
A point $x\in \mathbf{R}$ is said to be a \emph{transition point} of $f$ if $$x\in \partial Supp(f)=\overline{Supp(f)}\setminus Supp(f)$$

Our construction uses in an essential way the structure and properties of Thompson's group $F$.
We shall only describe the features of $F$ here that we need, and we direct the reader to \cite{CannonFloydParry} and \cite{Belk} for more comprehensive surveys.
Recall that the group $\textup{PL}^+([0,1])$ is the group of orientation preserving piecewise linear homeomorphisms of $[0,1]$.
Recall that $F$ is defined as the subgroup of $\textup{PL}^+([0,1])$ that satisfy the following:
\begin{enumerate}
\item Each element has at most finitely many breakpoints. All breakpoints lie in the set of dyadic rationals, i.e. $\mathbf{Z}[\frac{1}{2}]$.
\item For each element, the derivatives, wherever they exist, are powers of $2$.
\end{enumerate}
By \emph{breakpoint} we mean a point where the derivative does not exist.
For $r,s\in \mathbf{Z}[\frac{1}{2}]\cap[0,1]$ such that $r<s$, we denote by $F_{[r,s]}$ the subgroup of elements whose support lies in $[r,s]$.
The following are well known facts that we shall need.
The group $F$ satisfies the following:
\begin{enumerate}
\item $F$ is $2$-generated.
\item For each pair $r,s\in \mathbf{Z}[\frac{1}{2}]\cap[0,1]$ such that $r<s$, the group $F_{[r,s]}$ is isomorphic to $F$ and hence is also $2$-generated.
\item $F'$ is simple and consists of precisely the set of elements $g\in F$ such that $\overline{Supp(g)}\subset (0,1)$.
\end{enumerate}

An interval $I\subseteq [0,1]$ is said to be a \emph{standard dyadic interval}, if it is of the form $[\frac{a}{2^n},\frac{a+1}{2^{n}}]$
such that $a,n\in \mathbf{N}, a<2^n-1$.
The following are elementary facts about the action of $F$ on the standard dyadic intervals.

\begin{lem}\label{TransitiveStandard}
Let $I,J$ be standard dyadic intervals in $(0,1)$.
Then there is an element $f\in F'$ such that:
\begin{enumerate} 
\item $I\cdot f=J$.
\item $f\restriction I$ is linear.
\end{enumerate}
\end{lem}
\begin{lem}\label{TransitiveStandard2}
Let $I_1,I_2$ and $J_1,J_2$ be standard dyadic intervals in $(0,1)$ such that $$sup(I_1)<inf(I_2)\qquad sup(J_1)<inf(J_2)$$
Then there is an element $f\in F'$ such that:
\begin{enumerate}
\item $I_1\cdot f=J_1$ and $I_2\cdot f=J_2$.
\item $f\restriction I_1$ and $f\restriction I_2$ are linear.
\end{enumerate}
\end{lem}

We fix $\iota:[0,1]\to [0,1]$ as the unique orientation reversing isometry.
We say that an element $f\in F$ is \emph{symmetric}, if $f=\iota\circ f \circ \iota$.
We say that a set $I\subset [0,1]$ is \emph{symmetric} if $I\cdot \iota=I$.
Note that given any symmetric set $I$ with nonempty interior, we can find a symmetric element $f\in F'$ such that $Supp(f)\subset int(I)$.

\begin{defn}\label{H}
We fix an element $c_0\in F$ with the following properties:
\begin{enumerate}
\item The support of $c_0$ equals $(0,\frac{1}{4})$ and $x\cdot c_0>x$ for each $x\in (0,\frac{1}{4})$.
\item $c_0\restriction {(0,\frac{1}{16})}$ equals the map $t\to 2t$.
\end{enumerate}
Let $$c_1=\iota\circ c_0\circ \iota\qquad \nu_1=c_0c_1$$
Note that $\nu_1\in F$ is a symmetric element.
We define a subgroup $H$ of $F$ as $$H=\langle F',\nu_1\rangle$$
Finally, we fix $$\nu_2,\nu_3:[0,1]\to [0,1]$$ as the standard generators whose supports are contained in $(\frac{1}{16},\frac{15}{16})$ and that generate the group  $F_{[\frac{1}{16},\frac{15}{16}]}$.
\end{defn}

\begin{lem}\label{3gen}
$H$ is generated by $\nu_1,\nu_2,\nu_3$. $H'$ is simple and consists of precisely the set of elements of $H$ (or $F$) that are compactly supported in $(0,1)$.
In particular, $H'=F'$.
\end{lem} 

\begin{proof}
For the first claim, it suffices to show that any element $g\in F'$ can be expressed as a word in the generating set $\nu_1,\nu_2,\nu_3$.
To see this, note that there is an $n\in \mathbf{Z}$ such that $\nu_1^{-n} g \nu_1^n\in F_{[\frac{1}{16},\frac{15}{16}]}$.
The second claim follows from observing that $H'$ coincides with $F'$ which is simple.
It is apparent that any element $f\in H'$ is compactly supported in $(0,1)$.
Conversely, given any element $g\in F$ that is compactly supported in $(0,1)$, we know that $g\in F'$. And since $F'=F''$, it holds that $g\in F''$.
Since $F'\subset H$, it follows that $F''\subset H'$ and hence $g\in H'$. 
\end{proof}

To generalise our construction in Section \ref{continuum} to provide a family of continuum many isomorphism types, we shall use the following family of groups.
For each $\alpha\in (0,1)$, fix a homeomorphism $f_{\alpha}\in \textup{PL}^+([0,1])$ whose slope at $0$ is $\alpha$.
Let $\mathcal{N}$ denote the family of groups $$\mathcal{N}=\{\Gamma_{\alpha}=\langle F,f_{\alpha}\rangle\mid \alpha\in (0,1)\setminus \mathbf{Z}[\frac{1}{2}]\}$$
Note that the group $\Gamma_{\alpha}$ is $3$-generated and its abelianization is $\mathbf{Z}^3$, as can be seen from the homomorphism
provided by the germs at $0,1$.
The following is an observation of Nicol\'{a}s Matte Bon and we include a short proof here for completeness. 

\begin{prop}\label{PLLemma}
$\mathcal{N}$ consists of continuum many isomorphism types of groups.
\end{prop}

\begin{proof}
Assume by way of contradiction that there is an uncountable subset $I\subset (0,1)\setminus \mathbf{Z}[\frac{1}{2}]$ such that for all $\alpha,\beta\in I, \Gamma_{\alpha}\cong \Gamma_{\beta}$.
Note that the action of each $\Gamma_{\alpha}$ is locally dense, and so by Rubin's theorem for each $\alpha,\beta\in I$ there is a homeomorphism $\phi_{\alpha,\beta}:[0,1]\to[0,1]$
such that the map $$f\to \phi^{-1}_{\alpha,\beta}\circ f \circ \phi_{\alpha,\beta}\qquad f\in \Gamma_{\alpha}$$
induces an isomorphism between $\Gamma_{\alpha},\Gamma_{\beta}$.
(For the statement of Rubin's theorem and the definition of locally dense actions we refer the reader to Section $6.1$ of \cite{KimKoberdaLodha}.)

Next, note that the groups of germs at $0$ for $\Gamma_{\alpha},\Gamma_{\beta}$ are both isomorphic to $\mathbf{Z}^2$,
and can be viewed as the abelian groups of translations $$A_{\alpha}=\langle t\to t+1, t\to log_2(\alpha)\rangle\qquad A_{\beta}=\langle t\to t+1, t\to log_2(\beta)\rangle$$ 
Note that the homeomorphism $\phi_{\alpha,\beta}$ also induces a homeomorphism $\mathbf{R}\to \mathbf{R}$ on the germ at $0$, which in turn induces via a topological conjugacy, an isomorphism between $A_{\alpha},A_{\beta}$.
Using relative translation numbers and the fact that the group of automorphisms of $\mathbf{Z}^2$ is countable, we obtain a contradiction.
\end{proof}

\section{The construction}

We consider the additive group $\frac{1}{2}\mathbf{Z}=\{\frac{1}{2}k\mid k\in \mathbf{Z}\}$.
A \emph{labelling} is a map $$\rho:\frac{1}{2}\mathbf{Z}\to \{a,b,a^{-1},b^{-1}\}$$
which satisfies:
\begin{enumerate}
\item $\rho(k)\in\{a,a^{-1}\}$ for each $k\in \mathbf{Z}$.
\item $\rho(k)\in \{b,b^{-1}\}$ for each $k\in \frac{1}{2}\mathbf{Z}\setminus \mathbf{Z}$.
\end{enumerate}

We regard $\rho(\frac{1}{2}\mathbf{Z})$ as a bi-infinite word with respect to the usual ordering of the integers.
A subset $X\subseteq \frac{1}{2}\mathbf{Z}$ is said to be a \emph{block} if it is of the form $$\{k,k+\frac{1}{2},...,k+\frac{1}{2}n\}$$
for some $k\in \frac{1}{2}\mathbf{Z}, n\in \mathbf{N}$.
Note that each block is endowed with the usual ordering inherited from $\mathbf{R}$.
The set of blocks of $\frac{1}{2}\mathbf{Z}$ is denoted as $\mathbf{B}$.
To each block $X=\{k,k+\frac{1}{2},...,k+\frac{1}{2}n\}$, we assign a formal word $$W_{\rho}(X)=\rho(k)\rho(k+\frac{1}{2})...\rho(k+\frac{1}{2}n)$$
which is a word in the letters $\{a,b,a^{-1},b^{-1}\}$.

Recall that given a word $w_1...w_n$ in the letters $\{a,b,a^{-1},b^{-1}\}$, the formal inverse of the word is $w_n^{-1}...w_1^{-1}$.
The formal inverse of $W_{\rho}(X)$ is denoted as $W_{\rho}^{-1}(X)$.

A labelling $\rho$ is said to be \emph{quasi-periodic} if the following holds:
\begin{enumerate}
\item For each block $X\in \mathbf{B}$, there is an $n\in \mathbf{N}$ such that whenever $Y\in \mathbf{B}$ is a block of size at least $n$,
then $W_{\rho}(X)$ is a subword of $W_{\rho}(Y)$.
\item For each block $X\in \mathbf{B}$, there is a block $Y\in \mathbf{B}$ such that $W_{\rho}(Y)=W_{\rho}^{-1}(X)$.
\end{enumerate}
Note that by \emph{subword} in the above we mean a string of consecutive letters in the word.

A nonempty finite word $w_1...w_n$ for $w_i\in \{a,b,a^{-1},b^{-1}\}$ is said to be a \emph{permissible word} if $n$ is odd and the following holds.
For odd $i\leq n$, $w_i\in \{a,a^{-1}\}$ and for even $i\leq n$, $w_i\in \{b,b^{-1}\}$.


\begin{lem}\label{quasi-periodicLabellings}
Given any permissible word $w_1...w_m$, there is a quasi-periodic labelling $\rho$ of $\frac{1}{2}\mathbf{Z}$ and a block $X\in \mathbf{B}$ satisfying that $W_{\rho}(X)=w_1...w_m$.
\end{lem}

\begin{proof}
We shall define the quasi-periodic labelling in an inductive manner. At Step $n+1$ of the process,
we shall produce a labelling $\rho_{n+1}$ defined on a block $B_{n+1}\in \mathbf{B}$
such that $$B_n\subset B_{n+1}\qquad \rho_{n+1}\restriction B_n=\rho_n$$
and 
$$\bigcup_{n\in \mathbf{N}}B_n=\frac{1}{2}\mathbf{Z}$$
The required labelling $\rho$ is then the unique labelling on $\frac{1}{2}\mathbf{Z}$ which satisfies that $\rho\restriction B_n=\rho_n$ for each $n\in \mathbf{N}$.
First we fix $$B_0=\{0,\frac{1}{2},...,\frac{1}{2}(m-1)\}\qquad \rho_0(\frac{1}{2}(i-1))=w_i\text{ for }1\leq i\leq m$$

For $n\in \mathbf{N}$ that is even, once we have defined the pair $\rho_n,B_n$, we define $$B_{n+1}=\{k_n-\frac{1}{2}(l_{n}+1),k_n-\frac{1}{2}l_n,...,k_n-1,k_n-\frac{1}{2}\}\cup B_n$$ where $$l_n=|B_n|\qquad k_n=inf(B_n)$$
and $\rho_{n+1}$ is defined as follows.
\begin{enumerate}
\item $\rho_{n+1}\restriction B_n=\rho_n$.
\item $\rho_{n+1}(k_n-\frac{1}{2})=b$
\item $\rho_{n+1}(k_n-\frac{1}{2}l)=(\rho_{n}(k_n+\frac{1}{2}(l-2))^{-1}$ if $1< l \leq l_n+1$.
\end{enumerate}

For $n\in \mathbf{N}$ that is odd, once we have defined the pair $\rho_n,B_n$, we define $$B_{n+1}=\{k_n+\frac{1}{2},k_n+1,...,k_n+\frac{1}{2}(l_n+1)\}\cup B_n$$ where $$l_n=|B_n|\qquad k_n=sup(B_n)$$
and $\rho_{n+1}$ is defined as follows.
\begin{enumerate}
\item $\rho_{n+1}\restriction B_n=\rho_n$.
\item $\rho_{n+1}(k_n+\frac{1}{2})=b$
\item $\rho_{n+1}(k_n+\frac{1}{2}l)=(\rho_{n}(k_n-\frac{1}{2}(l-2))^{-1}$ if $1< l \leq l_n+1$.
\end{enumerate}

It is a straightforward exercise to verify that this is a quasi-periodic labelling with the required additional feature.
\end{proof}

The above proposition provides a systematic method for producing quasi-periodic labellings of $\frac{1}{2}\mathbf{Z}$.
To each labelling $\rho$, we shall associate a group $G_{\rho}<\textup{Homeo}^+(\mathbf{R})$ as follows.

\begin{defn}\label{SimpleGroup}
Let $H<\textup{Homeo}^+([0,1])$ be the group defined in Definition \ref{H}.
Recall from Lemma \ref{3gen} that the group $H$ is generated by the three elements $\nu_1,\nu_2,\nu_3$ defined in Definition \ref{H}.
In what appears below, by $\cong_T$ we mean that the restrictions are topologically conjugate via the unique orientation preserving isometry that maps $[0,1]$ to the respective interval.
We define the homeomorphisms $$\zeta_1,\zeta_2,\zeta_3,\chi_1,\chi_2,\chi_3:\mathbf{R}\to \mathbf{R}$$ as follows for each $i\in \{1,2,3\}$ and $n\in \mathbf{Z}$:
$$\zeta_i\restriction [n,n+1]\cong_{T}\nu_i \qquad \text{ if } \rho(n+\frac{1}{2})=b$$
$$\zeta_i\restriction [n,n+1]\cong_T(\iota\circ \nu_i  \circ \iota) \qquad \text{ if }\rho(n+\frac{1}{2})=b^{-1}$$
$$\chi_i\restriction [n-\frac{1}{2},n+\frac{1}{2}]\cong_T \nu_i \qquad \text{ if }\rho(n)=a$$
$$\chi_i\restriction [n-\frac{1}{2},n+\frac{1}{2}]\cong_T (\iota\circ \nu_i \circ \iota)\qquad \text{ if }\rho(n)=a^{-1}$$

The group $G_{\rho}$ is defined as $$G_{\rho}:=\langle \zeta_1,\zeta_2,\zeta_3,\chi_1,\chi_2,\chi_3\rangle<\textup{Homeo}^+(\mathbf{R})$$
We denote the above generating set of $G_{\rho}$ as $$\mathbf{S}_{\rho}:=\{\zeta_1,\zeta_2,\zeta_3,\chi_1,\chi_2,\chi_3\}$$
We also define subgroups $$\mathcal{K}:=\langle \zeta_1,\zeta_2,\zeta_3\rangle\qquad \mathcal{L}:=\langle \chi_1,\chi_2,\chi_3\rangle$$ of $G_{\rho}$ that are both isomorphic to $H$.
We fix the isomorphisms, defined by the above, as:
$$\lambda:H\to \mathcal{K}\qquad \pi:H\to \mathcal{L}$$
where for each $f\in H, n\in \mathbf{Z}$.
 $$\lambda(f)\restriction [n,n+1]\cong_T f \qquad \text {if }\rho(n+\frac{1}{2})=b$$
$$\lambda(f)\restriction [n,n+1]\cong_T (\iota\circ f\circ \iota)\qquad \text{ if }\rho(n+\frac{1}{2})=b^{-1}$$
$$\pi(f)\restriction [n-\frac{1}{2},n+\frac{1}{2}]\cong_T f\qquad \text {if }\rho(n)=a$$
$$\pi(f)\restriction [n-\frac{1}{2},n+\frac{1}{2}]\cong_T (\iota\circ f\circ \iota)\qquad \text{ if }\rho(n)=a^{-1}$$
We also denote the naturally defined inverse isomorphisms as:
$$\lambda^{-1}:\mathcal{K}\to H\qquad \pi^{-1}:\mathcal{L}\to H$$
Note that the definition of $\mathcal{K},\mathcal{L}$ requires us to fix a labelling $\rho$ but we denote them as such for simplicity of notation.
\end{defn}

\begin{remark}\label{OtherLabellings}
Note that the group $G_{\rho}$ above is defined for any labelling of $\frac{1}{2}\mathbf{Z}$.
In our proof of simplicity, it shall become apparent why quasi-periodicity of the labelling is required.
But in general one may consider arbitrary labellings.
For instance, consider the labelling $\tau$ which maps every element of $\mathbf{Z}$ to $a$ and every element of $\frac{1}{2}\mathbf{Z}\setminus \mathbf{Z}$ to $b$.
The associated group $G_{\tau}$ is then the lift of the standard action of Thompson's group $T<\textup{PL}^+(\mathbf{S}^1)=\textup{PL}^+(\mathbf{R}/\mathbf{Z})$ to the real line.
Hence there is a short exact sequence $$1\to \mathbf{Z}\to G_{\tau}\to T\to 1$$
Here $\mathbf{Z}$ is the center of the group $G_{\tau}$ which corresponds to the subgroup of integer translations.
The group $G_{\tau}$ admits the following global description. It is the group of piecewise linear homeomorphisms $f$ of the real line that satisfy:
\begin{enumerate}
\item The set $\{x\in \mathbf{R}\mid f'(x)\text{ is not defined}\}$ is discrete and is a subset of $\mathbf{Z}[\frac{1}{2}]$.
\item At each point $x\in \mathbf{R}$ where $f'(x)$ is defined, it is an integer power of $2$.
\item $f$ commutes with integer translations.
\end{enumerate}

For this labelling, the group admits a non trivial quotient onto $T$, and hence it is not simple. However, for any labelling $\rho$, $G_{\rho}$ is \emph{perfect}, i.e. $G_{\rho}'=G_{\rho}$.
\end{remark}

\begin{prop}\label{commutator}
Let $\rho$ be any labelling of $\frac{1}{2}\mathbf{Z}$.
Then $G_{\rho}=\langle \mathcal{K}'\cup\mathcal{L}'\rangle$.
It follows that $G_{\rho}=G_{\rho}'$. 
\end{prop}

\begin{proof}
Recall from Lemma \ref{3gen} that the group $H$ satisfies that $H'$ is simple and consists of precisely the set of elements that are compactly supported in $(0,1)$.
In particular, the generators $\nu_2,\nu_3$ lie in $H'$. 
It is clear from the definition that $\zeta_2,\zeta_3\in \mathcal{K}'$ and $\chi_2,\chi_3\in \mathcal{L}'$.
We shall observe that $\zeta_1\in \mathcal{L}'$ and $\chi_1\in \mathcal{K}'$.


Note that the closure of the support of $\zeta_1$ is contained in the set $$\bigcup_{n\in \mathbf{Z}} (n-\frac{1}{2},n+\frac{1}{2})$$
and $$\zeta_1\restriction (n-\frac{1}{2},n+\frac{1}{2})\cong_Tf$$ for a symmetric element $f\in H'$.
Since the restriction of $\zeta_1$ is symmetric for each such interval, we conclude that $\zeta_1$ lies in $\pi(H')$ and hence $\zeta_1\in \mathcal{L}'$.
Similarly, $\chi_1\in \mathcal{K}'$.
\end{proof}

\section{The core of the proof of simplicity}

In this section we shall reduce the proof of Theorem \ref{main} to Proposition \ref{mainprop}, which is formulated below and proved in Section $5$.
First, we take a small but important detour to remark the following.

\begin{prop}\label{freesubgroups}
For any labelling $\rho$, the group $G_{\rho}$ contains non abelian free subgroups.
\end{prop}

\begin{proof}
Let $f\in F'=H'$ be an element such that $Supp(f)=[\frac{1}{16},\frac{15}{16}]$ and $x\cdot f>x$ for each $x\in (\frac{1}{16},\frac{15}{16})$.
Moreover, by replacing $f$ by a power of $f$ if necessary, we also assume that $$[\frac{6}{16},\frac{10}{16}]\cdot f\subset (\frac{14}{16},\frac{15}{16})\qquad [\frac{6}{16},\frac{10}{16}]\cdot f^{-1}\subset (\frac{1}{16},\frac{2}{16})$$
Then the group $\langle \lambda(f),\pi(f)\rangle$ is free since the sets $$\bigcup_{n\in \mathbf{Z}} [n-\frac{2}{16},n+\frac{2}{16}]\qquad \bigcup_{n\in \mathbf{Z}} [n-\frac{1}{2}-\frac{2}{16},n-\frac{1}{2}+\frac{2}{16}]$$
form a ping pong table for the so called \emph{ping pong lemma}. (For the statement of this Lemma, we refer the reader to \cite{DeLaHarpe}.)
\end{proof}

We now develop some notation and define a certain family of actions which shall be useful in the proof.
In the rest of the article we shall fix a quasi-periodic labelling $\rho$. 

\begin{defn}
For every $X\subseteq \frac{1}{2}\mathbf{Z}\setminus \mathbf{Z}$, we define a group action $\mathcal{K}_{X}<\textup{Homeo}^+(\mathbf{R})$ by the representation 
$$\lambda_X:H\to \textup{Homeo}^+(\mathbf{R})$$
defined for each $f\in H$ as 
$$\lambda_X(f)\restriction [n,n+1]=\lambda(f)\restriction [n,n+1]\qquad \text{ if }n+\frac{1}{2}\in X$$
$$\lambda_X(f)\restriction [n,n+1]=id\restriction [n,n+1]\qquad \text{ if }n+\frac{1}{2}\notin X$$
Similarly, for every $X\subseteq \mathbf{Z}$, we define a group action $\mathcal{L}_{X}<\textup{Homeo}^+(R)$ by the representation 
$$\pi_X:H\to \textup{Homeo}^+(\mathbf{R})$$
defined for each $f\in H$ as 
$$\pi_X(f)\restriction  [n-\frac{1}{2},n+\frac{1}{2}]=\pi(f)\restriction [n-\frac{1}{2},n+\frac{1}{2}]\qquad \text{ if }n\in X$$
$$\pi_X(f)\restriction  [n-\frac{1}{2},n+\frac{1}{2}]=id\restriction  [n-\frac{1}{2},n+\frac{1}{2}]\qquad \text{ if }n\notin X$$

If $X\neq \emptyset$, there are naturally defined inverse isomorphisms which we denote as:
$$\lambda_X^{-1}:\mathcal{K}_X\to H\qquad \pi_X^{-1}:\mathcal{L}_X\to H$$
Note that  $\mathcal{K}_{\frac{1}{2}\mathbf{Z}\setminus \mathbf{Z}}=\mathcal{K}$ and $\mathcal{L}_{\mathbf{Z}}=\mathcal{L}$.
\end{defn}

These groups shall play an important role in the argument, and the following are some basic facts about them.

\begin{prop}\label{specialgroups}
The following holds.
\begin{enumerate}
\item For any $X,Y\subseteq \frac{1}{2}\mathbf{Z}\setminus \mathbf{Z}$, $$\mathcal{K}_{X\cup Y}'<\langle \mathcal{K}_X'\cup \mathcal{K}_Y'\rangle$$
\item For any $X,Y\subseteq \mathbf{Z}$, $$\mathcal{L}_{X\cup Y}'<\langle \mathcal{L}_X'\cup\mathcal{L}_Y'\rangle$$
\end{enumerate}
\end{prop}

\begin{proof}
We will prove the first part. 
The proof of the second part is similar.

If $X,Y$ are disjoint, then it is easy to see that $$\mathcal{K}_{X\cup Y}'< \langle \mathcal{K}_{X}'\cup\mathcal{K}_Y'\rangle$$
Also, if $X\subset Y$, then it is easy to see that $$\mathcal{K}_{Y\setminus X}'< \langle \mathcal{K}_{X}'\cup\mathcal{K}_Y'\rangle$$
{\bf Claim}: For any nonempty $X,Y\subseteq \mathbf{Z}$ it is true that $\mathcal{K}_{X\cap Y}'< \langle \mathcal{K}_{X}'\cup\mathcal{K}_Y'\rangle$.

{\bf Proof}: Given $g\in \mathcal{K}_{X\cap Y}'$, let $$f=\lambda_{X\cap Y}^{-1}(g)\in H'$$
Note that there is an $m\in \mathbf{N}$ such that $$\nu_1^{-m}f\nu_1^{m}\in F_{[\frac{1}{16},\frac{15}{16}]}'$$
We can find a word $W=w_1...w_n\in F_{[\frac{1}{16},\frac{15}{16}]}'$ in $\nu_2,\nu_3,\nu_2^{-1},\nu_3^{-1}$ such that $$w_1...w_n=\nu_1^{-m}f\nu_1^{m}$$
Moreover we assume that for each $i$, either $w_i=\nu_2^{l_i}$ or $w_i=\nu_3^{l_i}$ for for some $l_i\in \mathbf{Z}\setminus \{0\}$, and if $i<n$, $w_{i+1}$ is an integer power of a letter in $\{\nu_2,\nu_3\}$ that is different from the letter that $w_i$ is an integer power of.
Since $W\in F_{[\frac{1}{16},\frac{15}{16}]}'$, both the sum of powers of $\nu_2$ and the sum of powers of $\nu_3$ in $W$ equals $0$.
We assume that $n$ is even below, and if $n$ is odd then the expression below admits a suitable modification to establish the claim.
Then it follows that $$\lambda_X(\nu_1^{-m})(\lambda_{X}(w_1)\lambda_{Y}(w_2)\lambda_X(w_3)...\lambda_{Y}(w_n))\lambda_X(\nu_1^m)$$ 
$$=\lambda_{X\cap Y}(\nu_1^{-m}) (\lambda_{X\cap Y}(w_1)\lambda_{X\cap Y}(w_2)...\lambda_{X\cap Y}(w_n))\lambda_{X\cap Y}(\nu_1^m)$$ $$=g\in \mathcal{K}_{X\cap Y}'$$
since the homeomorphism corresponding to the word $$\lambda_X(\nu_1^{-m})(\lambda_{X}(w_1)\lambda_{Y}(w_2)\lambda_X(w_3)...\lambda_{Y}(w_n))\lambda_X(\nu_1^m)$$
is trivial on the intervals $$\{[n,n+1]\mid n+\frac{1}{2}\notin X\cap Y\}$$
Our claim follows.

It follows from the claim and the observations before the claim that $$\mathcal{K}_{X\setminus (X\cap Y)}'=\mathcal{K}_{X\setminus Y}'< \langle \mathcal{K}_{X}'\cup\mathcal{K}_Y'\rangle\qquad \mathcal{K}_{Y\setminus (X\cap Y)}'=\mathcal{K}_{Y\setminus X}'< \langle \mathcal{K}_{X}'\cup\mathcal{K}_Y'\rangle$$
and hence $$\mathcal{K}_{X\cup Y}'<\langle \mathcal{K}_X'\cup\mathcal{K}_Y'\rangle$$
\end{proof}

\begin{lem}\label{conjugation}
Let $N\leq G_{\rho}$ be a normal subgroup. Then the following holds.
\begin{enumerate}
\item If $N\cap \mathcal{K}_X'\neq \{id\}$ for some $X\subseteq \frac{1}{2}\mathbf{Z}\setminus \mathbf{Z}$, then $\mathcal{K}_X'<N$.
\item If $N\cap \mathcal{L}_X'\neq \{id\}$ for some $X\subseteq \mathbf{Z}$, then $\mathcal{L}_X'<N$.
\end{enumerate}
\end{lem}

\begin{proof}
We shall prove the first part. The proof of the second part is similar.
Let $g\in N\cap \mathcal{K}_X'$ be a nonidentity element.
Let $h=\lambda_X^{-1}(g)\in H'$.
Recall that $H'\cong \mathcal{K}_X'$ is simple, so it suffices to show that for any $h_1\in H'$, $\lambda_X(h_1^{-1} h h_1)\in N$.
This is true since $$\lambda(h_1^{-1})\lambda_X(h)\lambda(h_1)\in N$$ and $$\lambda(h_1^{-1})\lambda_X(h)\lambda(h_1)=\lambda_X(h_1^{-1})\lambda_X(h)\lambda_X(h_1)=\lambda_X(h_1^{-1} h h_1)$$
\end{proof}

The following Proposition is at the technical core of the article.
This will be proved in the subsequent section.

\begin{prop}\label{mainprop}
Given a nonidentity element $f\in G_{\rho}$, let $\langle \langle f\rangle \rangle_{G_{\rho}}$ be the normal closure of $f$ in $G_{\rho}$.
Then there are sets $$X_1,...,X_n\subset \frac{1}{2}\mathbf{Z}\setminus \mathbf{Z}\qquad Y_1,...,Y_m\subset \mathbf{Z}$$ such that:
\begin{enumerate}
\item $\bigcup_{1\leq i\leq n} X_i=\frac{1}{2}\mathbf{Z}\setminus \mathbf{Z}$ and $\langle \langle f\rangle \rangle_{G_{\rho}}\cap \mathcal{K}_{X_i}'\neq \{id\}\text{ for all }1\leq i\leq n$.
\item $\bigcup_{1\leq i\leq m} Y_i=\mathbf{Z}$ and $\langle \langle f\rangle \rangle_{G_{\rho}}\cap \mathcal{L}_{Y_j}'\neq \{id\}\text{ for all }1\leq j\leq m$.
\end{enumerate}

\end{prop}

\begin{proof}[Proof of Theorem \ref{main}]
$G_{\rho}$ is a finitely generated subgroup of $\textup{Homeo}^+(\mathbf{R})$ by construction.
After combining Propositions \ref{specialgroups}, \ref{conjugation} and \ref{mainprop} it follows that 
given a nonidentity element $f\in G_{\rho}$:
$$\mathcal{K}'\subseteq \langle \langle f\rangle \rangle_{G_{\rho}}\qquad \mathcal{L}'\subseteq  \langle \langle f\rangle \rangle_{G_{\rho}}$$
By Proposition \ref{commutator} we know that $G_{\rho}=\langle \mathcal{K}'\cup\mathcal{L}'\rangle$ and so it follows that $\langle \langle f\rangle \rangle_{G_{\rho}}=G_{\rho}$.
Therefore $G_{\rho}$ is simple.
\end{proof}

\section{Proof of Proposition \ref{mainprop}}

Our work in this section shall be dedicated to a proof of Proposition \ref{mainprop}.
We shall prove part $(1)$ of the Proposition. The other half of the Proposition admits a similar proof.

Throughout this section we shall assume that $\rho$ is a quasi-periodic labelling.
We shall develop some structural results concerning the group action of $G_{\rho}$ and various subgroups.
The proof of Proposition \ref{mainprop} will follow from Propositions \ref{SpecialElements1} and \ref{SpecialElements2}.
The section will be devoted to formulating and proving these propositions, and then finishing with the proof of \ref{mainprop}.
Recall that $\mathbf{S}_{\rho}$ is the generating set of $G_{\rho}$ as defined in Definition \ref{SimpleGroup}.

\begin{lem}\label{ElementsProperties}
Let $f\in G_{\rho}$ be a nonidentity element such that $$f=w_1...w_k\qquad w_i\in \mathbf{S}_{\rho}\text{ for }1\leq i\leq k$$
Then the following hold:
\begin{enumerate}
\item The set of breakpoints of $f$ is discrete and the set of transition points is also discrete.
\item $f$ fixes a point in $\mathbf{R}$.
\item For each $x\in \mathbf{R}$ and each $i\leq k$, 
$$x\cdot w_1...w_i\in [x-(k+1),x+(k+1)]$$
\end{enumerate}
\end{lem}

\begin{proof}
By construction, $G_{\rho}$ is a subgroup of the group of piecewise linear homeomorphisms of the real line each of whose elements has a countable, discrete set of breakpoints.
Note that every element of $G_{\rho}$ will have a discrete set of transition points since each non trivial affine map has at most one fixed point in the real line,
and the restriction of the action of an element of $G_{\rho}$ on a compact interval is piecewise linear with finitely many breakpoints.
Recall from the definition of the generating set that for each generator $w\in \mathbf{S}_{\rho}$
and each $x\in \mathbf{R}$, it holds that $|x\cdot w-x|<1$.
In particular, if $\tilde{w}_i$ is the partial map that is the restriction of $w_i$ on $[x-(k+1),x+(k+1)]$, then for each $i\leq k$ $$x\cdot \tilde{w}_1...\tilde{w}_i$$ is defined and $$x\cdot w_1...w_i=x\cdot \tilde{w}_1...\tilde{w_i}$$
This proves the third statement.

Now we shall prove the second statement.
Assume that there is an $x\in \mathbf{R}$ such that $x\cdot f>x$.
(A symmetric argument works if we start with an $x\in \mathbf{R}$ such that $x\cdot f<x$.)
Consider the block $$B_1=\{y-(k+1),...,y-\frac{1}{2},y,y+\frac{1}{2},...,y+(k+1)\}\qquad y=\left \lfloor{x}\right \rfloor$$
Note that by the observation above, the restriction of the action of $f$ on $[y,y+1]$ only depends on the restriction of $\rho$ on this block.
Since $\rho$ is quasi-periodic, there is another block $$B_2=\{z-(k+1),...,z-\frac{1}{2},z,z+\frac{1}{2},...,z+(k+1)\}\qquad \text{ for some }z\in \mathbf{Z}$$
such that $W_{\rho}^{-1}(B_1)=W_{\rho}(B_2)$.
It follows that there is an $$x_1\in (z-(k+1),z+(k+1))$$ such that $x_1\cdot f<x_1$.
Since $f$ is a homeomorphism, by the intermediate value theorem it follows that $f$ fixes a point in any compact interval containing the points $x,x_1$.
\end{proof}

\begin{remark}
Note that in the proof of the second part of Lemma \ref{ElementsProperties}, we use part $(2)$ of the definition of a quasi-periodic labelling.
This part of the lemma fails for the non-simple group discussed in Remark \ref{OtherLabellings} since it contains integer translations as elements.
\end{remark}

\begin{lem}\label{minimal0}
The action of $G_{\rho}$ on $\mathbf{R}$ is minimal.
\end{lem}

\begin{proof}
The action of $H$ on $(0,1)$ is minimal.
So the restrictions $$\mathcal{K}\restriction (n,n+1)\qquad \mathcal{L}\restriction (n-\frac{1}{2},n+\frac{1}{2})$$
are minimal for each $n\in \mathbf{Z}$.
It is easy to see that this implies the minimality of the group $G_{\rho}$ which is generated by $\mathcal{K}$ and $\mathcal{L}$.
\end{proof}

\begin{lem}\label{minimal}
For each pair of elements $m_1,m_2\in \mathbf{Z}$ and a closed interval $I\subset (m_1,m_1+1)$,  
there is a word $w_1...w_k$ in the generators $\mathbf{S}_{\rho}$ such that $$I\cdot w_1...w_k\subset (m_2,m_2+1)$$
and $$I\cdot w_1...w_i\subset [m_3,m_4]$$ for each $1\leq i\leq k$, where $m_3=inf\{m_1,m_2\}, m_4=sup\{m_1+1,m_2+1\}$.
\end{lem}

\begin{proof}
We assume that $m_1<m_2$. The case $m_2<m_1$ admits a similar proof.
Let $x=m_1+\frac{1}{2}$.
Recall that the restrictions $$\mathcal{K}\restriction (t,t+1)\qquad \mathcal{L}\restriction (t-\frac{1}{2},t+\frac{1}{2})$$
are minimal for each $t\in \mathbf{Z}$.
Using this fact, we find a word $v_1...v_n$ in $\mathbf{S}_{\rho}$ such that $$x\cdot v_1...v_n\subset (m_2,m_2+1)$$
and $$x\cdot v_1...v_i\subset [m_3,m_4]$$ for each $1\leq i\leq n$.
Since our maps are continuous, we find an interval $J\subset (m_1,m_1+1)$ such that $x\in J$ and the above statement holds if one replaces $x$ by $J$.

Recall that for any pair of closed intervals $I_1,I_2\subset (0,1)$ there is an element $f\in H'=F'$ such that $I_1\cdot f\subset I_2$.
In particular, this holds for the restriction of the action of $\mathcal{K}$ on $(m_1,m_1+1)$.
So we find a word $u_1...u_l$ in $\{\zeta_1,\zeta_2,\zeta_3\}$ such that $I\cdot u_1...u_l\subset J$. 
Then the required word is $w_1...w_k=(u_1...u_l)(v_1...v_n)$.
\end{proof}

We fix the natural map $$\Phi:\mathbf{R}\to [0,1)\qquad \Phi(x):=x-\left \lfloor{x}\right \rfloor$$
The following are elementary corollaries of the third part of Lemma \ref{ElementsProperties}, and we leave the proof of the first (which follows from the definitions) to the reader.

\begin{cor}\label{Phiimages}
Let $f\in G_{\rho}$.
There is an $m\in \mathbf{N}$ such that
for any $x_1,x_2 \in \mathbf{R}$ so that $x_1-x_2\in \mathbf{Z}$, the following holds.
If the restriction of $\rho$ on the blocks $$\{y_1-m,...,y_1-\frac{1}{2},y_1,y_1+\frac{1}{2},...,y_1+m\}\qquad y_1=\left \lfloor{x_1}\right \rfloor$$
and $$\{y_2-m,...,y_2-\frac{1}{2},y_2,y_2+\frac{1}{2},...,y_2+m\}\qquad y_2=\left \lfloor{x_2}\right \rfloor$$
is equal, then $\Phi(x_1\cdot f)=\Phi(x_2\cdot f)$.
\end{cor}

\begin{cor}\label{FiniteImagePhi}
Let $f\in G_{\rho}$. Then the set $$\{\Phi(x\cdot f)\mid x\in \mathbf{Z}\}$$
is finite.
\end{cor}

\begin{proof}
From the previous Corollary it follows that there is an $m\in \mathbf{Z}$ such that for any $y\in \mathbf{Z}$, $\Phi(y\cdot f)$ is determined by the restriction of the labelling $\rho$ to the block $$\{y-m,...,y-\frac{1}{2},y,y+\frac{1}{2},...,y+m\}$$
Since there are finitely many words in $\{a,b,a^{-1},b^{-1}\}$ of length $4m+1$, it follows that the set $$\{\Phi(x\cdot f) \mid x\in \mathbf{Z}\}$$
is finite.
\end{proof}

\begin{defn}\label{Z}
We denote by $\mathcal{Z}<G_{\rho}$ as the subgroup of elements of $G_{\rho}$ which fix $\mathbf{Z}$ pointwise and for which no point in $\mathbf{Z}$ is a transition point.
This subgroup shall play a useful role in the rest of the section.
Note that $\mathcal{Z}$ contains $\mathcal{K}'$ as a subgroup.
\end{defn}

\begin{lem}\label{SpecialElements}
Let $f\in G_{\rho}$ be a nonidentity element.
Then $\langle \langle f\rangle \rangle_{G_{\rho}}\cap \mathcal{Z}\neq \{id\}$.
\end{lem}

\begin{proof}
{\bf Claim}: There is an element $h\in G_{\rho}$ such that $$Supp(h)\cap (\mathbf{Z}\cup \{x\cdot f\mid x\in \mathbf{Z}\})=\emptyset$$
Moreover, the element $h$ can be chosen so that it does not commute with $f$.

{\bf Proof of claim}: From the second part of Lemma \ref{ElementsProperties} we know that $f$ fixes a point in $\mathbf{R}$.
Since $f$ is nonidentity, we can find an open interval $J$ which is a component of the support of $f$ with a boundary point $z\in \mathbf{R}$ which is fixed by $f$.
We define $B=\{\Phi(x\cdot f)\mid x\in \mathbf{Z}\}$ and
$$X=\Phi(J)\cup (\Phi(J)\cdot \iota) \qquad Y=B\cup (B\cdot \iota)$$
(Here recall that we denote by $\iota$ the unique orientation reversing isometry of $[0,1]$.)
Note that $X\setminus Y$ is a symmetric subset of $[0,1]$ by definition and has nonempty interior since $\Phi(J)$ has nonempty interior and $B$ is a finite set (thanks to Corollary \ref{FiniteImagePhi}).
Let $g\in H'$ be a nonidentity symmetric element such that $\overline{Supp}(g)\subset X\setminus Y$.
Let $h=\lambda(g)$.

First, we claim that since $\partial Supp(h)\cap J\neq \emptyset$ it follows that $h,f$ do not commute.
Let $x$ be a boundary point of $Supp(h)$ that lies in $J$.
If $f,h$ were to commute, then the elements of the set $\{f^n(x)\mid n\in \mathbf{Z}\}$ will all be transition points of $h$,
with an accumulation point in $\mathbf{R}$ which is another transition point of $h$.
This contradicts the first part of Lemma \ref{ElementsProperties}.

Next, observe that by design $\overline{Supp(h)}\cap \{x\cdot f\mid x\in \mathbf{Z}\}=\emptyset$ and $h\in \mathcal{Z}$.  
This implies that $$f h f^{-1} h^{-1}\in (\langle \langle f\rangle \rangle_{G_{\rho}}\cap \mathcal{Z})\setminus \{id\}$$
\end{proof}

For each $q\in \frac{1}{2}\mathbf{Z}\setminus \mathbf{Z}$, we define the map 
$$\kappa_q:\mathcal{Z}\to H\qquad \kappa_q(f)\cong_T f\restriction [q-\frac{1}{2},q+\frac{1}{2}]\text{ if } \rho(q)=b$$
$$\kappa_q:\mathcal{Z}\to H\qquad \kappa_q(f)\cong_T \iota\circ f\circ \iota\restriction [q-\frac{1}{2},q+\frac{1}{2}]\text{ if } \rho(q)=b^{-1}$$

Given an element $f\in \mathcal{Z}$, the set $$\{\kappa_q(f)\mid q\in \frac{1}{2}\mathbf{Z}\setminus \mathbf{Z}\}\setminus \{id\}$$
is called the \emph{set of atoms} of $f$.
An element $h\in H\setminus \{id\}$ is said to be an atom of $f\in \mathcal{Z}$ if $$h\in \{\kappa_q(f)\mid q\in \frac{1}{2}\mathbf{Z}\setminus \mathbf{Z}\}$$
If $h\in H$ is an atom of $f\in \mathcal{Z}$, the set $$\{q\in \frac{1}{2}\mathbf{Z}\setminus \mathbf{Z}\mid \kappa_q(f)=h\}$$
is called the \emph{set of flags} of $h$ in $f$.
Also, an element of this set shall be referred to as a flag of $h$ in $f$.
In this language, an element $f\in \mathcal{Z}$ satisfies that $f\in \mathcal{K}_X\setminus \{id\}$ if and only if $f$ has only one atom and the set of its flags in $f$ is $X$.

\begin{lem}\label{type}
For each $f\in \mathcal{Z}$, the set of atoms of $f$ is finite. 
\end{lem}

\begin{proof}
Let $f=w_1...w_k$ where each $w_i\in \mathbf{S}_{\rho}$.
From the third part of Lemma \ref{ElementsProperties} it follows that the restriction $f\restriction [n,n+1]$ for any $n\in \mathbf{Z}$ is determined by the restriction of the labelling $\rho$ to the block $$\{n-(k+1),...,n-\frac{1}{2},n,n+\frac{1}{2},...,n+(k+1)\}$$
Since there are finitely many words in $\{a,b,a^{-1},b^{-1}\}$ of length $4(k+1)+1$, our conclusion follows.
\end{proof}

\begin{lem}\label{SpecialElements0}
Let $f\in \mathcal{Z}$ be a nonidentity element, and let $h\in H$ be an atom of $f$ with set of flags $X\subset \frac{1}{2}\mathbf{Z}\setminus \mathbf{Z}$.
Then there is an $m\in \mathbf{N}$, and a partition of $\frac{1}{2}\mathbf{Z}$ into blocks $\{B_i\mid i\in \mathbf{Z}\}$ such that $|B_i|=m$ and $B_i\cap X\neq \emptyset$ $\text{ for each }i\in \mathbf{Z}$.
\end{lem}

\begin{proof}
Fix $n\in X$.
Thanks to Lemma \ref{ElementsProperties}, we know that there is a $k\in \mathbf{N}$ such that  $f\restriction [n-\frac{1}{2},n+\frac{1}{2}]$ is determined by the restriction of the labelling $\rho$ to the block $$B=\{n-k-\frac{1}{2},n-k,...,n,n+\frac{1}{2},...,n+k+\frac{1}{2}\}$$
By the first part of the definition of a quasi-periodic labelling, we may fix $m\in \mathbf{N}$ such that the following holds. 
For any block $B'$ of length $m$, there is a block $B''\subset B'$ such that $W_{\rho}(B'')=W_{\rho}(B)$.
It follows that for each such pair $B',B''$, there is an $$l\in (\frac{1}{2}\mathbf{Z}\setminus \mathbf{Z})\cap B''\subseteq (\frac{1}{2}\mathbf{Z}\setminus \mathbf{Z})\cap B'$$ such that $$f\restriction [l-\frac{1}{2},l+\frac{1}{2}]\cong_T f\restriction [n-\frac{1}{2},n+\frac{1}{2}]$$ 
and $\rho(n)=\rho(l)$.
In particular, $l\in X$.
Therefore, our conclusion holds for any partition of $\frac{1}{2}\mathbf{Z}$ into blocks of length $m$.
\end{proof}

\begin{remark}
The above Lemma is where the first part of the definition of quasi-periodic is required.
Recall that the second part of the definition was needed in the proof of Lemma \ref{ElementsProperties}.
\end{remark}

\begin{prop}\label{SpecialElements1}
Let $f\in \mathcal{Z}$ be a nonidentity element.
Then there are elements $$f_1,...,f_n\in \mathcal{Z}\cap \langle \langle f\rangle \rangle_{G_{\rho}}$$
such that each $f_i$ has an atom $h_i\in H'$ with the associated set of flags $X_i$ satisfying $$\bigcup_{1\leq i\leq n} X_i=\frac{1}{2}\mathbf{Z}\setminus \mathbf{Z}$$
\end{prop}

\begin{proof}
Let $h$ be an atom of $f$ and let $X$ be the set of flags of $h$ in $f$.
Let $\{h_1,...,h_n\}$ be the set of atoms of $f$ that are not equal to $h$.
Applying Lemma \ref{SpecialElements0}, we obtain a partition of $\frac{1}{2}\mathbf{Z}$ into blocks $\{B_i\mid i\in \mathbf{Z}\}$
such that $|B_i|=m$ and $B_i\cap X\neq \emptyset$ $\text{ for each }i\in \mathbf{Z}$.
Since there are finitely many words of length $m$ in the letters $\{a,b,a^{-1},b^{-1}\}$,
the set of words $\{W_{\rho}(B_i)\mid i\in \mathbf{Z}\}$
is finite.

Let $j, B_i$ be a pair such that $j\in B_i\cap \frac{1}{2}\mathbf{Z}\setminus \mathbf{Z}$ and 
$\kappa_j(f)=h$.
Let $[s,s+1]$ be an interval where $s,s+1\in \mathbf{Z}\cap B_i$ and let $$I=\overline{Supp(f)}\cap [j-\frac{1}{2},j+\frac{1}{2}]$$
Note that $I$ is contained in a closed subinterval of $(j-\frac{1}{2},j+\frac{1}{2})$.
Using Lemma \ref{minimal}, find an element $$g=w_1...w_n\in G_{\rho}\qquad w_1,...,w_l\in \mathbf{S}_{\rho}$$ such that 
$$I\cdot g\subset [s,s+1]$$ and  $$I\cdot w_1...w_i\subset (inf(B_i),sup(B_i))\text{ for each }
1\leq i\leq l$$

Recall from Corollary \ref{FiniteImagePhi} that the set $$S_0=\{\Phi(z\cdot g^{-1})\mid z\in \mathbf{Z}\}$$ is finite and so the set $$S_1=S_0\cup (S_0\cdot \iota)$$
is also finite.
Next, we define $$S_2=\bigcup_{p\in \{h_1,...,h_n,h\}}(S_1\cdot p^{-1}\cup S_1\cdot (\iota\circ p^{-1} \circ \iota))$$
Note that $S=S_1\cup S_2$ is a finite set. 
So we can find an element $k\in H'$ such that the support of $k$ has exactly one component, denoted by $I_1$, whose closure is contained in $Supp(h)\setminus S$.

Since $\overline{Supp(k)}\subset Supp(h)$, and since $k$ has a finite set of transition points, it follows that $$h^{-1} k^{-1} hk\neq id\qquad Supp(h^{-1} k^{-1} hk)\subset Supp(h)$$
Additionally, since $\lambda(k)\in \mathcal{Z}$, it follows that $$q=f^{-1} \lambda(k)^{-1} f \lambda(k)\in (\langle \langle f\rangle \rangle_{G_{\rho}}\cap \mathcal{Z})\setminus \{id\}$$
Moreover, by design $q$ has the property that $q$ has an atom with flag at $j$ and $$Supp(q)\cap [j-\frac{1}{2},j+\frac{1}{2}]\subset I$$

Next, we observe that $$\overline{Supp(\lambda(k))}\cap (\mathbf{Z}\cup \mathbf{Z}\cdot g^{-1}\cup (\mathbf{Z}\cdot g^{-1}\cdot f^{-1}))=\emptyset$$
So it follows that $$g^{-1} f^{-1} \lambda(k)^{-1} f \lambda(k) g=g^{-1} q g\in (\langle \langle f\rangle \rangle_{G_{\rho}}\cap \mathcal{Z})\setminus \{id\}$$
and by design $g^{-1}qg$ has an atom with a flag at the interval $[s,s+1]$.

Note that in the above process the elements $k,g$ are determined by the fixed element $f$, the labelling on the block $B_i$ of a fixed length $m$ and a choice of $j,s\in B_i$.
Since there are only finitely many possible labellings on any given $B_i$ of length $m$, and finitely many choices of $j,s\in B_i$,
our conclusion follows.
\end{proof}

Let $h\in H' \setminus \{id\}$ and $h_1,...,h_n\in H'$.
The pair $$h, (h_1,...,h_n)$$ is said to be \emph{compatible} if $h$ admits a transition point $x\in (0,1)$ such that $x$ is not a transition point of any element in the set $\{h_i\mid h_i\neq h\}$.
Note that in this definition we treat $(h_1,...,h_n)$ as a multiset (i.e., with possibly several occurrences of the same element in the set).
This is done for the sake of flexibility in the arguments that follow.
Let $g\in \mathcal{Z}$ be a nonidentity element and let $h_1,...,h_n\in H$ be the atoms of $g$. 
A pair $h_i,g$ is said to be \emph{$\mathcal{Z}$-compatible}, if the pair $$h_i,(h_1,...,h_n)$$
is compatible.

\begin{lem}\label{confluence}
Let $h,h_1,...,h_n\in H'$ be elements such that $h\neq id$.
Then either the pair $$h, (h_1,...,h_n)$$ is compatible, or there is an element $g\in H'$ such that the pair 
$$[h^{-1} g h, g], ([h_1^{-1} g h_1, g],...,[h_n^{-1} g h_n, g])$$
is compatible.
\end{lem}

\begin{proof}
Let $x$ be a transition point of $h$. We find a sufficiently small interval of the form $$I=(x-\epsilon, x)\subset (0,1)$$ 
such that whenever $x$ is not a transition point of $h_i$ (for any $1\leq i\leq n$), then either $h_i$ fixes $I$ pointwise, or $I\cdot h_i\cap I=\emptyset$.
In particular, for such an $h_i$ and for any element $g\in H'$ whose support lies in $I$, it holds that $$[h_i^{-1} g h_i, g]=id$$
Moreover, we assume that $h$ moves each point in $I$. 
(Note that to ensure this it may be necessary to work with an interval of the form $(x,x+\epsilon)$,
and in this case we can argue similarly.)
In this proof we shall find a $g$ such that $Supp(g)\subset I$ and that moreover satisfies the statement of the Lemma.

Recall that the set of transition points for any element of $G_{\rho}$ is discrete.
It follows that there is an $\epsilon_1\in (0,\epsilon)$ such that the following holds. 
For any $h_i\in \{h_1,...,h_n\}$, if $x$ is a transition point of $h_i$,
it holds that $h_i\restriction (x-\epsilon_1,x)$ is either the identity or an affine map of the form $$t\to s_i(t-x)+x\qquad \text{ with slope }s_i\in \mathbf{R}_{>0}$$
For each such $h_i$, the slope $s_i$ determines this map.
Moreover, we make our choice so that additionally $h\restriction (x-\epsilon_1,x)$ is also of this form and we denote the slope of $h\restriction (x-\epsilon_1,x)$ by $s$.
We assume without loss of generality that $s<1$ (the case $s>1$ can be dealt with in a similar fashion).

We find an element $g\in H'=F'$ and dyadic intervals $I_1,I_2\subset (x-\epsilon_1,x)$ with $sup(I_1)<inf(I_2)$ such that the following holds:
\begin{enumerate}
\item $g$ has precisely two components of support which are $I_1,I_2$.
\item $I_1\cdot h\subset I_2$, and hence $$\partial Supp([h^{-1} g h, g])\cap I_2\neq \emptyset$$
\item $I_2\cdot h\cap I_2=\emptyset$ and hence $$\partial Supp([h^{-1} g h, g])\cap (I_2\cdot h)= \emptyset$$
\item For each $1\leq i\leq n$ such that $x$ is a transition point of $h_i$, either $$s_i\in \{s,\frac{1}{s}\}\qquad \text{ or }\qquad ((I_1\cup I_2)\cdot h_i)\cap (I_1\cup I_2)=\emptyset$$ 
\end{enumerate}
It is easy to check that by design the pair $$[h^{-1} g h, g], ([h_1^{-1} g h_1, g],...,[h_n^{-1} g h_n, g])$$
is compatible.
\end{proof}

\begin{cor}\label{confluence1}
Let $f\in \mathcal{Z}$ be a nonidentity element.
Let $X$ be the set of flags of an atom $h$ of $f$.
Then either the pair $h,f$ is $\mathcal{Z}$-compatible, or there is an element $k\in \langle \langle f\rangle \rangle_{G_{\rho}}\cap \mathcal{Z}$
with an atom $l\in H'$ of $k$ with a set of flags $Y$ such that:
\begin{enumerate}
\item $X\subseteq Y$.
\item The pair $l,k$ is $\mathcal{Z}$-compatible.
\end{enumerate}
\end{cor}

\begin{proof}
Let $h,h_1,...,h_n\in H'$ be the atoms of $f$.
We apply Lemma \ref{confluence} to the pair $$h, (h_1,...,h_n)$$
to obtain an element $g\in H'$ such that $$[h^{-1} g h, g], ([h_1^{-1} g h_1, g],...,[h_n^{-1} g h_n, g])$$
is compatible.

Define $$k=[f^{-1}\lambda(g) f, \lambda(g)]$$
Note that $k\in \langle \langle f\rangle \rangle_{G_{\rho}}$ since $$k=\lambda(g) ((\lambda(g)^{-1} f^{-1} \lambda(g)) f (\lambda(g) f^{-1} \lambda(g)^{-1}) f) \lambda(g)^{-1}$$
It follows immediately from the definition of $\lambda$ and our hypothesis that the pair $[h^{-1} g h, g], k$ is $\mathcal{Z}$-compatible, and the set of flags of $[h^{-1} g h, g]$
contains $X$.
\end{proof}

\begin{lem}\label{confluence2}
Let $f\in \mathcal{Z}$ be a nonidentity element and $h$ be an atom of $f$.
Let $X$ be the set of flags of $h$ in $f$.
If the pair $h,f$ is $\mathcal{Z}$-compatible, then there is a nonidentity element $g\in \langle \langle f\rangle \rangle_{G_{\rho}}\cap \mathcal{K}_X'$.
\end{lem}

\begin{proof}
Let $h,h_1,...,h_n\in H'$ be the atoms of $f$.
Let $x\in (0,1)$ be a transition point of $h$ such that for each $h_i\in \{h_1,...,h_n\}$, $x$ is not a transition point of $h_i$.
Then we can choose a small dyadic interval $$I=(x-\epsilon, x+\epsilon)\subset (0,1)$$ and an element $l\in H$ such that:
\begin{enumerate}
\item $l$ has one component of support and $Supp(l)=I$.
\item For each $1\leq i\leq n$, either $I\cdot h_i\cap I=\emptyset$ or $h_i$ fixes each point in $I$.
In particular, it holds that $$[h_i^{-1} l h_i, l]=id$$
\item Either $h$ moves every point in $[x-\epsilon,x)$ or $h$ moves every point in $(x,x+\epsilon]$.
In particular, it holds that $$[h^{-1} l h, l]\neq id$$
\end{enumerate}
Note that the third requirement is possible to achieve since the set of transition points of $h$ is finite.
The required element is then $$g=[f^{-1} \lambda(l) f, \lambda(l)]$$
Note that $g\in \mathcal{K}_X'$ by design and $g\in \langle \langle f\rangle \rangle_{G_{\rho}}$ since $$g=\lambda(l) ((\lambda(l)^{-1} f^{-1} \lambda(l)) f (\lambda(l) f^{-1} \lambda(l)^{-1}) f) \lambda(l)^{-1}$$
\end{proof}

\begin{prop}\label{SpecialElements2}
Let $f\in \mathcal{Z}$ be a nonidentity element.
Let $h\in H$ be an atom of $f$ and let $X$ be the set of flags of $h$ in $f$.
Then there is a set $Y\subset \frac{1}{2}\mathbf{Z}\setminus \mathbf{Z}$ such that $X\subseteq Y$ and $\mathcal{K}_Y'\cap \langle \langle f\rangle \rangle_{G_{\rho}}\neq \{id\}$.
\end{prop}

\begin{proof}
Applying Corollary \ref{confluence1}, we obtain an element $f_1\in \langle \langle f\rangle \rangle_{G_{\rho}}\cap \mathcal{Z}$
with an atom $h_1\in H'$ with a set of flags $Y$ such that:
\begin{enumerate}
\item $X\subseteq Y$.
\item The pair $h_1,f_1$ is $\mathcal{Z}$-compatible.
\end{enumerate}
(Unless the pair $f,h$ is $\mathcal{Z}$-compatible, in which case we set $f_1=f,h_1=h$.)
Next, we apply Lemma \ref{confluence2} to obtain a nonidentity element $$g\in \langle \langle f_1\rangle \rangle_{G_{\rho}}\cap \mathcal{K}_Y'\subseteq \langle \langle f\rangle \rangle_{G_{\rho}}\cap \mathcal{K}_Y'$$
\end{proof}

\begin{proof}[Proof of Proposition \ref{mainprop}]
Let $g\in G_{\rho}$ be a nonidentity element.
Then from Lemma \ref{SpecialElements} it follows that there is a nonidentity element $f\in \mathcal{Z}\cap \langle \langle g\rangle \rangle_{G_{\rho}}$.
From Proposition \ref{SpecialElements1} it follows that there are elements $$g_1,...,g_n\in \langle \langle f\rangle \rangle_{G_{\rho}}\cap \mathcal{Z} \qquad h_1,...,h_n\in H$$ such that each $h_i$ is an atom of $g_i$ with flags $X_i$ such that $$\bigcup_{1\leq i\leq n}X_i=\frac{1}{2}\mathbf{Z}\setminus \mathbf{Z}$$
Applying Proposition \ref{SpecialElements2} to each triple $h_i,g_i,X_i$, we obtain sets $Y_1,...,Y_n$ and nonidentity elements $l_1,...,l_n$ such that:
\begin{enumerate}
\item $l_i\in \langle \langle g_i\rangle \rangle_{G_{\rho}}\subseteq \langle \langle f\rangle \rangle_{G_{\rho}}\subseteq \langle \langle g\rangle \rangle_{G_{\rho}}$ for each $1\leq i\leq n$.
\item $l_i\in \mathcal{K}_{Y_i}'$ for each $1\leq i\leq n$.
\item $X_i\subseteq Y_i$ and hence  $$\bigcup_{1\leq i\leq n}Y_i=\frac{1}{2}\mathbf{Z}\setminus \mathbf{Z}$$
\end{enumerate}
\end{proof}

\section{Uncountably many isomorphism types}\label{continuum}

In this section we shall prove Theorem \ref{continuumsimple}.
The construction of the groups $G_{\rho,\alpha}$ for a given $\alpha\in (0,1]\setminus \mathbf{Z}[\frac{1}{2}]$
will be obtained in exactly the same fashion as the construction of $G_{\rho}$ with the exception that the group $H$ in the construction
shall be replaced by an overgroup $H_{\alpha}$ defined as follows.
Recall from the preliminaries that $$\mathcal{N}=\{\Gamma_{\alpha}=\langle F,f_{\alpha}\rangle\mid \alpha\in (0,1]\setminus \mathbf{Z}[\frac{1}{2}]\}$$
where $f_{\alpha}\in \textup{PL}^+([0,1])$ is a chosen element whose slope at $0$ is $\alpha$.

\begin{defn}
Let $\mathcal{I}$ be the set of standard dyadic subintervals of $(0,1)$.
(Recall that a standard dyadic interval is always closed.)
Using Lemma \ref{TransitiveStandard}, for each pair $I,J\in \mathcal{I}$,
we fix an element $f_{I,J}\in F'$ such that $I\cdot f_{I,J}=J$ and $f_{I,J}\restriction I$ is linear.
We make the choices such that $f_{I,J}^{-1}=f_{J,I}$.

Fix an interval $I\in \mathcal{I}$.
Recall that $F_I$ is the subgroup of $F$ consisting of elements supported in $I$,
and that $F_I\cong F$.
Let $\sigma_{\alpha} \in \textup{PL}^+([0,1])$ be a homeomorphism that satisfies:
\begin{enumerate}
\item $\langle \sigma_{\alpha},F_I\rangle \restriction I\cong \Gamma_{\alpha}$.
\item $\sigma_{\alpha}\restriction [0,1]\setminus I=id$.
\end{enumerate}

Next, we fix a standard dyadic interval $J$ in $(0,1)$ such that
$sup(J)<inf(I)$.
Let $\xi_{\alpha}$ be a homeomorphism of $[0,1]$ defined as:
$$\xi_{\alpha}\restriction I=\sigma_{\alpha}\qquad \xi_{\alpha}\restriction J=f_{J,I}\sigma_{\alpha}^{-1} f_{I,J}\restriction J\qquad \xi_{\alpha}\restriction [0,1]\setminus (I\cup J)=id$$
Note that since the abelianization of each $\Gamma_{\alpha}\in \mathcal{N}$ is $\mathbf{Z}^3$ (as discussed in the preliminaries),
it follows that $\langle \xi_{\alpha}, F_I\rangle \cong \Gamma_{\alpha}$. 
Finally, we define the group $H_{\alpha}$ as:
$$H_{\alpha}=\langle H, \xi_{\alpha} \rangle$$
Since the group $$\langle F_I, \xi_{\alpha}\rangle\cong \Gamma_{\alpha}$$ it follows that the group $H_{\alpha}$ contains $\Gamma_{\alpha}$ as an abstract subgroup.
\end{defn}

Now we are ready to define the group $G_{\rho,\alpha}$.

\begin{defn}
We define the homeomorphisms $$\gamma_{\alpha},\omega_{\alpha}:\mathbf{R}\to \mathbf{R}$$ as follows for each $n\in \mathbf{Z}$:
$$\gamma_{\alpha}\restriction [n,n+1]\cong_{T}\xi_{\alpha} \qquad \text{ if } \rho(n+\frac{1}{2})=b$$
$$\gamma_{\alpha}\restriction [n,n+1]\cong_T(\iota\circ \xi_{\alpha}  \circ \iota) \qquad \text{ if }\rho(n+\frac{1}{2})=b^{-1}$$
$$\omega_{\alpha}\restriction [n-\frac{1}{2},n+\frac{1}{2}]\cong_T \xi_{\alpha} \qquad \text{ if }\rho(n)=a$$
$$\omega_{\alpha}\restriction [n-\frac{1}{2},n+\frac{1}{2}]\cong_T (\iota\circ \xi_{\alpha} \circ \iota)\qquad \text{ if }\rho(n)=a^{-1}$$

The group $G_{\rho,\alpha}$ is defined as $$G_{\rho,\alpha}=\langle G_{\rho}, \gamma_{\alpha},\omega_{\alpha}\rangle$$
\end{defn}

\begin{remark}
Recall from Proposition \ref{freesubgroups} that the group $G_{\rho}$ contains free subgroups.
Since $G_{\rho,\alpha}$ contains $G_{\rho}$ as a subgroup, it also contains free subgroups.
\end{remark}

Another tool we shall require for the generalization is the following theorem, due to Higman (See \cite{Higman}).
Let $\Gamma$ be a group of bijections of some set $E$. 
(The reader may specialise this as $E=[0,1]$ and regard $\Gamma$ as a group of homeomorphisms of $E$). 

\begin{thm}\label{Higman}
(Higman's simplicity criterion) Suppose that for all $f_1,f_2,f_3\in \Gamma\setminus \{1_{\Gamma}\}$, there is an $f_4\in \Gamma$ such that:
$$S\cdot f_4\cdot f_3 \cap S\cdot f_4=\varnothing\qquad  \text{where } S=Supp(f_1)\cup Supp(f_2)$$
Then $\Gamma'$ is simple.
\end{thm}

\begin{lem}\label{perfect}
For each $\alpha\in (0,1]\setminus \mathbf{Z}[\frac{1}{2}]$, the group $H_{\alpha}$ satisfies the following.
\begin{enumerate}
\item $H_{\alpha}'$ is simple and consists of precisely the set of elements of $H_{\alpha}$ that are compactly supported in $(0,1)$.
\item Each element in $H_{\alpha}$ has finitely many breakpoints and finitely many transition points.
\end{enumerate}
\end{lem}

\begin{proof}
Let $H_{\alpha,c}$ be the subgroup of $H_{\alpha}$ consisting of the elements in $H_{\alpha}$ that are compactly supported in $(0,1)$.
Note that this is generated by $H',\xi_{\alpha}$.
We apply Higman's simplicity criterion above to $H_{\alpha,c}$ to obtain that the derived subgroup of this group is simple.
(The proof is straightforward, since $F'\subset H_{\alpha,c}$).
It suffices to show that $H_{\alpha,c}$ is perfect, and so it suffices to show that the generator $\xi_{\alpha}$ is expressible as a commutator of elements in the group.
Let $I_1$ be a standard dyadic interval in $(0,1)$ such that $sup(I)<inf(I_1)$.
Using Lemma \ref{TransitiveStandard2} we find elements $f_1,f_2\in F'\subset H_{M,c}$ such that:
$$J\cdot f_1=J\qquad I\cdot f_1=I_1$$
and $f_1\restriction J, f_1\restriction I$ are linear.
 $$I\cdot f_2=J\qquad I_1\cdot f_2=I$$
and $f_2\restriction I, f_2\restriction I_1$ are linear.
The required word is then $$\xi= f_2^{-1} (f_1^{-1} \xi_{\alpha} f_1 \xi_{\alpha}^{-1}) f_2$$
The second statement of the Lemma holds since it holds for any subgroup of $\textup{PL}^+([0,1])$.
\end{proof}

\begin{proof}[Proof of theorem \ref{continuumsimple}]
To prove the theorem, we simply replace $H$ by $H_{\alpha}$ in the construction defined in the previous sections (with suitable modifications for the definitions of $\lambda,\pi,\kappa_n, \mathcal{Z}$ etc.).
For the proof of simplicity, note that $H<H_{\alpha}<\textup{PL}^+([0,1])$ and the only additional ingredient needed in the proofs is the statement of Lemma \ref{perfect}.
Therefore, one can simply replace $H$ by $H_{\alpha}$ in the proof of simplicity.
Recall that a countable group contains countably many finitely generated subgroups.
Since $\mathcal{N}$ consists of continuum many isomorphism types of finitely generated groups (from Proposition \ref{PLLemma} in the preliminaries), and $\Gamma_{\alpha}<G_{\rho,\alpha}$,
we conclude the statement of the theorem.
\end{proof}

\end{document}

%% file: macros.tex
\newtheorem{thm}{Theorem}[section]

\newtheorem{lem}[thm]{Lemma}
\newtheorem*{lem*}{Lemma}

\newtheorem{prop}[thm]{Proposition}

\newtheorem{cor}[thm]{Corollary}

\newtheorem{question}[thm]{Question}

\theoremstyle{remark}
\newtheorem{remark}[thm]{Remark}

\theoremstyle{definition}
\newtheorem{defn}[thm]{Definition}

\newcounter{my_enumerate_counter}

\newcommand\comment[1]{}

\renewcommand{\epsilon}{\varepsilon}